\documentclass[11pt]{article}

\usepackage{amsmath}
\usepackage{amsthm}
\usepackage{amsfonts}
\usepackage{bbm,xcolor}
\usepackage{hyperref}

\usepackage{ulem}

\newcommand{\me}{\mathbb{E}}
\newcommand{\mr}{\mathbb{R}}

\newcommand{\mn}{\mathbb{N}}
\newcommand{\mmp}{\mathbb{P}}

\DeclareMathOperator{\1}{\mathbbm{1}}

\newtheorem{thm}{Theorem}[section]
\newtheorem{lemma}[thm]{Lemma}

\newtheorem{example}[thm]{Example}
\newtheorem{prop}[thm]{Proposition}
\theoremstyle{definition}

\theoremstyle{remark}
\newtheorem{rem}[thm]{Remark}

\newcommand{\eee}{{\rm e}}

\begin{document}

\title{Limit theorems for decoupled renewal processes}\date{}
\author{Congzao Dong\footnote{School of Mathematics and Statistics, Xidian University, Xi’an, China; e-mail address: czdong@xidian.edu.cn} \ \ Iryna Feshchenko\footnote{Faculty of Computer Science and Cybernetics, Taras Shevchenko National University of Kyiv, 01601 Kyiv, Ukraine; e-mail address: irynafeshchenko@knu.ua}  \ \ and \ \ Alexander Iksanov\footnote{Faculty of Computer Science and Cybernetics, Taras Shevchenko National University of Kyiv, 01601 Kyiv, Ukraine; e-mail address:
iksan@univ.kiev.ua}}
\maketitle
\begin{abstract}
\noindent The decoupled standard random walk is a sequence of independent random variables $(\hat S_n)_{n\geq 1}$, in which $\hat S_n$ 
has the same distribution as the position at time $n$ of a standard random walk with nonnegative jumps. Denote by $\hat N(t)$ the number of elements of the decoupled standard random walk which do not exceed $t$. The random process $(\hat N(t))_{t\geq 0}$ is called decoupled renewal process. Under the assumption that $t\mapsto \mmp\{\hat S_1>t\}$ 
is regularly varying at infinity of nonpositive index larger than $-1$ we prove a functional central limit theorem in the Skorokhod space equipped with the $J_1$-topology for the decoupled renewal processes, properly scaled, centered and normalized. Also, under the assumption that $t\mapsto \mmp\{\hat S_1>t\}$ 
is regularly varying at infinity of index $-\alpha$, $\alpha\in [0,1)\cup (1,2)$ or the distribution of $\hat S_1$ 
belongs to the domain of attraction of a normal distribution we prove a law of the iterated or single logarithm for $\hat N(t)$, again properly normalized and centered. As an application, we obtain a law of the single logarithm for the number of atoms of a determinantal point process with the Mittag-Leffler kernel, which lie in expanding discs.
\end{abstract}


\hphantom{2000 Mathematics Subject Classification: } Secondary: 60G50

\section{Introduction}

Let $\xi_1$, $\xi_2,\ldots$ be independent copies of a nonnegative random variable $\xi$ with a nondegenerate distribution. Put $S_n=\xi_1+\ldots+\xi_n$ for $n\in\mn:=\{1,2,\ldots\}$ and then $N(t):=\sum_{n\geq 1}\1_{\{S_n\leq t\}}$ for $t\geq 0$. The random sequence $(S_n)_{n\geq 1}$ is called {\it standard random walk} with nonnegative jumps and the random process $(N(t))_{t\geq 0}$ is called {\it renewal process}. Let $\hat S_1$, $\hat S_2,\ldots$ be independent random variables such that, for each $n\in\mn$, $\hat S_n$ has the same distribution as $S_n$. Put $\hat N(t):=\sum_{n\geq 1}\1_{\{\hat S_n\leq t\}}$ for $t\geq 0$. Following \cite{Alsmeyer+Iksanov+Kabluchko:2025}, we call the sequence $(\hat S_n)_{n\geq 1}$ {\it decoupled standard random walk} and the process $(\hat N(t))_{t\geq 0}$ {\it decoupled renewal process}. 

Our interest in decoupled standard random walks has been raised by their recent appearance in connection with particular determinantal point processes. Now we provide some details. Let $\mathbb{C}$ denote the set of complex numbers. For $\rho>0$, define the kernel $C_\rho$ by $$C_\rho(z,w)=\frac{\rho}{2\pi}{\rm E}_{2/\rho,\,2/\rho}(z\bar w)\eee^{-|z|^\rho/2-|w|^\rho/2},\quad z,w\in \mathbb{C}.$$ Here, $\bar w$ denotes the complex conjugate of $w$ and, for $a,b>0$, ${\rm E}_{a,\,b}$ denotes the Mittag-Leffler function with parameters $a$ and $b$ given by $${\rm E}_{a,\,b}(z):=\sum_{k\geq 0}\frac{z^k}{\Gamma(ak+b)},\quad z\in \mathbb{C},$$ and $\Gamma$ is the Euler gamma-function. Denote by $\Theta_\rho$ a simple point process on $\mathbb{C}$ such that, for any $k\in\mn$ and any mutually disjoint Borel subsets $B_1,\ldots, B_k$ of $\mathbb{C}$, $$\me \Big[ \prod_{j=1}^k \Theta_\rho(B_j)\Big]=\int_{B_1\times\ldots\times B_k} {\rm det}(C_\rho(z_i, z_j))_{1\leq i,j\leq k}\,{\rm d}z_1\ldots {\rm d}z_k,$$ where ${\rm det}$ denotes the determinant. The point process $\Theta_\rho$ is a {\it determinantal point process} with kernel $C_\rho$ with respect to Lebesgue measure on $\mathbb{C}$, see \cite{Hough+Krishnapur+Peres+Virag:2009}. The process $\Theta_2$ (which corresponds to $\rho=2$) is known in the literature as the {\it infinite Ginibre point process}, see Sections 4.3.7 and 4.7 in the book \cite{Hough+Krishnapur+Peres+Virag:2009} and Part I of the monograph \cite{Byun+Forrester:2025}.

The set of absolute values of atoms of $\Theta_\rho$ has the same distribution as $((\hat S_n)^{1/\rho})_{n\geq 1}$, where $\hat S_1$ has the gamma distribution with parameters $2/\rho$ and $1$, that is,
\begin{equation}\label{eq:gamma}
\mmp\{\hat S_1\in{\rm d}x\}=\frac{1}{\Gamma(2/\rho)}x^{2/\rho-1}\eee^{-x}\1_{(0,\infty)}(x){\rm d}x,
\end{equation}
see pp.~3-4 in \cite{Adhikari:2018}. For each $t\geq 0$, let $\Theta_\rho(D_t)$ denote the number of atoms of $\Theta_\rho$ inside the disk $D_t:=\{z\in \mathbb{C}: |z|<t\}$. Then
\begin{equation}\label{eq:principal}
(\Theta_\rho(D_t))_{t\geq 0}~~\text{has the same distribution as}~~(\hat N(t^\rho))_{t\geq 0}=\Big(\sum_{n\geq 1}\1_{\{\hat S_n\leq t^\rho\}}\Big)_{t\geq 0}
\end{equation}
with $\hat S_1$ as in \eqref{eq:gamma}.

Proposition 1.4 in \cite{Fenzl+Lambert:2022} is a functional central limit theorem (FCLT) in the Skorokhod space for $(\Theta_2(D_t))_{t\geq 0}$, properly scaled, centered, and normalized. Equivalently, it is an FCLT for the decoupled renewal process $(\hat N(t))_{t\geq 0}$, which corresponds to $\xi$ having the exponential distribution of unit mean. Theorem 2.1 in \cite{Alsmeyer+Iksanov+Kabluchko:2025} is an FCLT for $(\hat N(t))_{t\geq 0}$ under the assumption that the distribution of $\xi$ belongs to the domain of attraction of a stable distribution with finite mean (cases (A), (B) and (C) in Section \ref{sect:prelim}). In the present article we prove an FCLT for $(\hat N(t))_{t\geq 0}$ under the assumption that the right distribution tail of $\xi$ is regularly varying at $\infty$ of nonpositive index larger than $-1$ (case (D) in Section \ref{sect:prelim}). As far as tightness of distributions was concerned (recall that an FLT is equivalent to weak convergence of finite-dimensional distributions plus tightness), Theorem 2.1 in \cite{Alsmeyer+Iksanov+Kabluchko:2025} required the assumption that the function $t\mapsto \sum_{n\geq 1}\mmp\{S_n\leq t\}$ is Lipschitz continuous on $[0,\infty)$. Here, we show that this assumption can be dispensed with.

Also, we prove a law of the single logarithm (LSL) for $\hat N(t)$ under the assumption that the distribution of $\xi$ belongs to the domain of attraction of a stable distribution with finite mean and a law of the iterated logarithm (LIL) under the assumption that the right distribution tail of $\xi$ is regularly varying at $\infty$ of nonpositive index larger than $-1$. These results are derived by an application of Theorem 1.6 in \cite{Buraczewski+Iksanov+Kotelnikova:2025}, which provides sufficient conditions ensuring that an LSL or an LIL holds for infinite sums of independent indicators parameterized by $t$, as $t\to\infty$.

\section{Main results}

\subsection{Preliminaries}\label{sect:prelim}

In what follows $\ell$ denotes a function slowly varying at $\infty$, that is, $\lim_{t\to\infty}(\ell(\lambda t)/\ell(t))=1$ for each $\lambda>0$. As usual, $f(t)\sim g(t)$ as $t\to A$ will mean that $\lim_{t\to A}(f(t)/g(t))=1$.

Put $\tau(t):=\inf\{k\geq 1: S_k>t\}$ for $t\geq 0$ and observe that $\tau(t)=N(t)+1$ almost surely (a.s.). We recall known facts concerning distributional convergence of $\tau(t)$, properly normalized, with or without centering. According to Proposition A.1 in \cite{Gnedin etal:2009} and a remark following it,
\begin{equation}\label{eq:basic conv}
\frac{\tau(t)-\mu^{-1} t}{\mu^{-1-1/\alpha}c_\alpha(t)}~{\overset{{\rm d}}\longrightarrow}~ Z_\alpha,\quad t\to\infty,
\end{equation}
where $\mu:=\me [\xi]<\infty$, provided that

\noindent (A) $\sigma^2={\rm Var}[\xi]\in (0,\infty)$; in which case $\alpha=2$, $c_2(t)=\sigma t^{1/2}$, and $Z_2$ is a random variable with the standard normal distribution;

\noindent (B) $\sigma^2=\infty$ and $\me[\xi^2 \1_{\{\xi \leq t\}}] \sim \ell(t)$ as $t \to \infty$; in which case $\alpha=2$, $c_2$ is any positive function such that $$\lim_{t\to\infty}\frac{t\ell (c_2(t))}{(c_2(t))^2}=1.$$

\noindent (C) $\mathbb{P}\{\xi > t\} \sim t^{-\alpha}\ell(t)$ as $t\to\infty$ for some $\alpha \in (1,2)$; in which case $c_\alpha$ is any positive function such that $$\lim_{t\to\infty} \frac{t\ell(c_\alpha(t))}{(c_\alpha(t))^\alpha} =1,$$ and $Z_\alpha$ is a random variable having the spectrally negative $\alpha$-stable distribution with the characteristic function
\begin{multline*}
\mathbb{E} [\exp({\rm i} v Z_\alpha)] \ = \ \exp\{|v|^\alpha (\alpha-1)^{-1}\Gamma(2-\alpha)(\cos(\pi\alpha/2)+{\rm i}\sin(\pi\alpha/2){\rm sign}\,(v))\},\quad v\in\mr,
\end{multline*}
where $\Gamma$ is the Euler gamma-function. It is known that, for $\alpha\in (1,2]$, the functions $c_\alpha$ do exist and vary regularly at $\infty$ of index $1/\alpha$, see, for instance, Lemma 6.1.3 on p.~209 in \cite{Iksanov:2016}. 

According to the aforementioned Proposition A.1 in \cite{Gnedin etal:2009},
\begin{equation}\label{eq:basic conv2}
\mmp\{\xi>t\} \tau(t)~{\overset{{\rm d}}\longrightarrow}~ Z_\alpha,\quad t\to\infty
\end{equation}
provided that

\noindent (D) $\mathbb{P}\{\xi > t\} \sim t^{-\alpha}\ell(t)$ as $t\to\infty$ for some $\alpha \in [0,1)$; in which case $Z_\alpha$ has the Mittag-Leffler distribution with parameter $\alpha$ (exponential of unit mean, if $\alpha=0$) characterized by its moments $$\me [(Z_\alpha)^n]=\frac{n!}{(\Gamma(1-\alpha))^n\Gamma(1+\alpha n)},\quad n\in\mn.$$

We refrain from discussing here the scenario where $\mmp\{\xi > t\} \sim t^{-1}\ell(t)$ as $t\to\infty$.

\subsection{Functional central limit theorems}

Throughout this section we treat the case (D). Put $G(t):=1/\mmp\{\xi>t\}$ for $t\geq 0$. For $\alpha\in [0,1)$, denote by $h_\alpha$ any nonnegative nondecreasing function satisfying $G(h_\alpha(t)) \sim \eee^t$ as $t\to\infty$. If $\alpha\in (0,1)$, the existence of such a function is secured by Theorem 1.5.12 in \cite{Bingham+Goldie+Teugels:1989}. If $\alpha=0$, put $h_0(t)=G^\leftarrow(\eee^t)$, where $G^\leftarrow(t):=\inf\{t\geq 0: G(t)>s\}$ for $s>1$. In the second part of the proof of Theorem 1.5.12 in \cite{Bingham+Goldie+Teugels:1989} the relation $f(f^\leftarrow(t))\sim t$ as $t\to\infty$ is derived for $f$ regularly varying at $\infty$ of positive index. However, that proof works equally well for slowly varying $f$. By this reasoning, we obtain $G(h_0(t))=G(G^\leftarrow(\eee^t))\sim \eee^t$ as $t\to\infty$.

Put $V(t):=\me [\hat N(t)]=\sum_{k\geq 1}\mmp\{\hat S_k\leq t\}=\sum_{k\geq 1}\mmp\{S_k\leq t\}$ for $t\geq 0$. Thus, $V+1$ is a renewal function. Let $I$ be an interval, finite or infinite. Denote by $D(I)$ the Skorokhod space of c\`{a}dl\`{a}g functions defined on $I$. We assume that the Skorokhod space is endowed with the $J_1$-topology, see \cite{Billingsley:1968}. We write $\Longrightarrow$ for weak convergence in a function space. As usual, $x\wedge y=\min (x,y)$ and $x\vee y=\max (x,y)$ for $x,y\in\mr$.

\begin{thm}\label{main}
Suppose (D). Then $$\big(\eee^{-t/2}\big(\hat N(h_\alpha(t+u))-V(h_\alpha(t+u)))\big)_{u\in\mr}\ \Longrightarrow \ X_\alpha,\quad t\to\infty$$ 
on $D(R)$, where $X_{\alpha}=(X_{\alpha}(u))_{u\in\mr}$ is a centered 
Gaussian process with covariance function
\begin{equation*} 
{\rm Cov}\,(X_{\alpha}(u),X_{\alpha}(v))\ =\ \int_0^\infty \mmp\{Z_\alpha> \eee^{-(u\wedge v)}y\}\mmp\{Z_\alpha\le\eee^{-(u\vee v)y}\}{\rm d}y
\end{equation*}
for $u,v\in\mr$.
\end{thm}
\begin{rem}
For $\alpha\in [0,1)$, put $Y_\alpha(u):= \eee^{-u/2}X_\alpha(u)$ for $u\in\mr$. The process $(Y_\alpha(u))_{u\in\mr}$ is stationary Gaussian. The process $(Y_0(u))_{u\in\mr}$ with the covariance $${\rm Cov}\,(Y_0(u), Y_0(v))=\frac{1}{2\cosh(u-v)},\quad u,v\in\mr$$ and its time-changed version arise naturally in various contexts, see \cite{Kabluchko:2019} and references therein, and \cite{Buraczewski+Dovgay+Iksanov:2020, Iksanov+Kabluchko:2018, Iksanov+Nikitin+Samoilenko:2020}.
\end{rem}

\subsection{Laws of the iterated or single logarithm}

Below we refer to conditions (A)-(D) introduced in Section \ref{sect:prelim}.

\begin{thm}\label{main1}
Suppose (A). Then
\begin{equation*} 
\limsup_{t\to \infty} \frac{\hat{N}(t) - V(t)}{t^{1/4}(\log t)^{1/2}} = \left( \frac{\sigma^2}{\mu^3 \pi} \right)^{1/4}\quad \text{{\rm a.s.}},
\end{equation*}
where $\mu=\me [\xi]<\infty$ and $\sigma^2={\rm Var}[\xi]\in (0,\infty)$.

\noindent Suppose (B). Then
\begin{equation*}
\limsup_{t\to \infty} \frac{\hat{N}(t) - V(t)}{(c_2(t) \log t)^{1/2}} =\left( \frac{1}{\mu^3 \pi} \right)^{1/4}\quad \text{{\rm a.s.}}
\end{equation*}

\noindent Suppose (C). Then
\begin{multline*}
\limsup_{t\to \infty} \frac{\hat{N}(t) - V(t)}{(c_\alpha (t)\log t)^{1/2}}= \left(\frac{2(\alpha - 1)\Gamma(1-1/\alpha)}{\alpha \mu^{1+1/\alpha}\pi}\left(\frac{-2\Gamma(2-\alpha)\cos(\pi \alpha/2)}{\alpha-1}\right)^{1/\alpha}\right)^{1/2}\quad \text{{\rm a.s.}},
		\end{multline*}
where $\Gamma$ is the Euler gamma-function.

\noindent Suppose (D). Then
\begin{equation}\label{eq:lil heavy}
\limsup_{t\to \infty}\Big(\frac{\mmp\{\xi>t\}}{\log\log t}\Big)^{1/2}(\hat{N}(t) - V(t))= \left( 2\int_0^\infty \mmp\{Z_\alpha > y\}\mmp\{Z_\alpha \leq y\}{\rm d}y \right)^{1/2}\quad \text{{\rm a.s.}}
\end{equation}

All the limit relations hold true, with $-\liminf$ replacing $\limsup$.
\end{thm}

\begin{rem}
In the case (D) with $\alpha=0$, a simplification is possible: $\int_0^\infty \mmp\{Z_0> y\}\mmp\{Z_0 \leq y\}{\rm d}y=1/2$. In particular, the right-hand side in \eqref{eq:lil heavy} is then equal to $1$.
\end{rem}

\begin{example}
Assume that $\xi$ has the gamma distribution with parameters $2/\rho$ and $1$, see formula~\eqref{eq:gamma}. Then $\mu=\sigma^2=2/\rho$. A specialization of the case (A) of Theorem \ref{main1} together with formula ~\eqref{eq:principal} ensures that
\begin{equation*}
\limsup_{t\to \infty} \frac{\Theta_\rho(D_t) - V(t^\rho)}{t^{\rho/4}(\log t)^{1/2}} = \frac{\rho}{(4\pi)^{1/4}}\quad \text{{\rm a.s.}}
\end{equation*}
and
\begin{equation*}
\liminf_{t\to \infty} \frac{\Theta_\rho(D_t) - V(t^\rho)}{t^{\rho/4}(\log t)^{1/2}} =-\frac{\rho}{(4\pi)^{1/4}}\quad \text{{\rm a.s.}}
\end{equation*}
\end{example}

\section{Proof of Theorem \ref{main}}

We start with a couple of preparatory results.
\begin{lemma}\label{lem:cov}
Suppose (D). Then
\begin{align*}
\lim_{t\to\infty} \eee^{-t}({\rm Cov}\,(\hat N(h_\alpha(t+u)), \hat N(h_{\alpha}(t+v))))=\int_0^\infty \mmp\{Z_\alpha>\eee^{-(u\wedge v)}y\}\mmp\{Z_\alpha \le \eee^{-(u\vee v)}y\}\ {\rm d}y
\end{align*}
for $u,v\in\mr$.
\end{lemma}
\begin{proof}
For convenience, define $S_0$ to be $0$. Recalling that $h_\alpha$ is a nondecreasing function, write, for $u<v$,
\begin{multline*}
{\rm Cov}\,(\hat N(h_\alpha(t+u)), \hat N(h_{\alpha}(t+v)))\\= \me \Big[\sum_{k\ge 1}(\1_{\{\hat S_{k}\le h_\alpha(t+u)\}}-\mmp\{\hat S_{k}\le h_\alpha(t+u)\})\sum_{j\ge 1}(\1_{\{\hat S_{j}\le h_\alpha(t+v)\}}-\mmp\{\hat S_{j}\le h_\alpha(t+v)\})\Big]\\=\int_0^\infty\mmp\{S_{\lfloor x\rfloor}\le h_\alpha (t+u)\}\mmp\{S_{\lfloor x\rfloor}>h_\alpha(t+v)\}\ {\rm d}x.
\end{multline*}
Changing the variable $x=\eee^t y$ and using the equality $\{S_k\leq z\}=\{\tau(z)>k\}$ which holds for $k\in\mn$ and $z\geq 0$ we infer
\begin{multline*}
{\rm Cov}(\hat N(h_\alpha(t+u)), \hat N(h_\alpha(t+v)))=\eee^t \int_0^\infty \mmp\{S_{\lfloor \eee^t y\rfloor}\le h_{\alpha}(t+u)\}\mmp\{S_{\lfloor \eee^t y\rfloor}>h_{\alpha}(t+v)\}{\rm d}y\\ =\eee^t \int_0^\infty \mmp\{\tau(h_{\alpha}(t+u))>\lfloor \eee^t y\rfloor\}\mmp\{\tau(h_{\alpha}(t+v))\le \lfloor \eee^t y\rfloor\}{\rm d}y. 
\end{multline*}
Formula \eqref{eq:basic conv2} entails that, for any fixed $y\geq 0$ and $w\in\mr$,
\begin{equation*} 
\lim_{t\to\infty}\mmp\{\tau(h_\alpha(t+w))>\lfloor \eee^t y\rfloor\}=\mmp\{Z_\alpha>\eee^{-w}y\}.
\end{equation*}
Since the function $y\mapsto \mmp\{Z_\alpha\leq y\}$ is continuous on $[0,\infty)$, the convergence is actually uniform in $y\in [0,\infty]$ (the compactified nonnegative halfline). Therefore,
$$\lim_{t\to\infty}\int_0^\infty \mmp\{\tau(h_{\alpha}(t+u))>\lfloor \eee^t y\rfloor\}\mmp\{\tau(h_{\alpha}(t+v))\le \lfloor \eee^t y\rfloor\}{\rm d}y=\int_0^\infty\mmp\{Z_\alpha>\eee^{-u}y\}\mmp\{Z_\alpha \leq \eee^{-v}y\}{\rm d}y.$$

The proof of Lemma \ref{lem:cov} is complete.
\end{proof}

For $t \geq 0$, put
\begin{equation}\label{def_v}
W(t):= \int_0^t V(t-y)\eee^{-y}{\rm d} y =\me [V(t - \eta)\1_{\{\eta \leq t\}}],
\end{equation}
where $\eta$ is a random variable with the exponential distribution of unit mean.
\begin{lemma}\label{uv_mon}
Pick $\lambda>0$ such that $\mmp\{\xi\leq \lambda\}>0$. Then $W$ is strictly increasing and continuous on $(\lambda, \infty)$.
\end{lemma}
\begin{proof}
Continuity of $W$ on $[0,\infty)$ is obvious.

By assumption, $V(t)>0$ for $t \geq \lambda$. For such $t$, $$\int_0^t V(y)\eee^y {\rm d}y \leq V(t)(\eee^t-1)<V(t)\eee^t$$ and thereupon $W(t)=\eee^{-t}\int_0^t V(y)\eee^y{\rm d}y < V(t)$. Fix any $h>0$ and write
\begin{multline*}
W(t+h)= \eee^{-(t+h)}\int_0^{t+h} V(y)\eee^y {\rm d}y = \eee^{-h} \eee^{-t}\int_0^t V(y)\eee^y {\rm d}y+\eee^{-(t+h)}\int_t^{t+h} V(y)\eee^y {\rm d}y \\ = \eee^{-h}W(t) +\eee^{-(t+h)}\int_t^{t+h}V(y)\eee^y {\rm d}y.
\end{multline*}
Finally,
\begin{multline*}
W(t+h)-W(t)=(\eee^{-h}-1)W(t)+\eee^{-(t+h)}\int_t^{t+h}V(y)\eee^y {\rm d}y \\ \geq (\eee^{-h}-1)W(t)+\eee^{-(t+h)}(\eee^{t+h}-\eee^t)V(t)= (1-\eee^{-h})(V(t)-W(t)) > 0.
\end{multline*}
Thus, $W$ is indeed strictly increasing on $(\lambda, \infty)$.

\end{proof}

\begin{proof}[Proof of Theorem \ref{main}]
Our proof follows the standard pattern. We prove weak convergence of finite-dimensional distributions and then check tightness.

For large $t$, put
$$ Z(t,u)\,:\,=\eee^{-t/2}\big(\hat N(h_{\alpha}(t+u))-V(h_{\alpha}(t+u))\big), \quad u\in\mr.$$
By the Cram\'{e}r-Wold device (see, for instance, pp.~48-49 in \cite{Billingsley:1968}), the weak convergence of its finite-dimensional distributions is equivalent to
\begin{equation}\label{eq:Cramer-Wold device}
\sum_{i=1}^{k}\lambda_i Z(t,u_{i})~{\overset{{\rm d}}\longrightarrow}~ \sum_{i=1}^{k} \lambda_{i} X_{\alpha}(u_{i}),\quad  t\to\infty,
\end{equation}
where ${\overset{{\rm d}}\longrightarrow}$ denotes weak convergence of one-dimensional distribution, for all $k\in\mn$, all $\lambda_{1},\ldots, \lambda_{k}\in\mr$ and all $-\infty<u_{1}<\ldots<u_{k}<\infty$. In view of $$\sum_{i=1}^{k}\lambda_i Z(t,u_{i})= \eee^{-t/2} \sum_{n\ge 1}\sum_{i=1}^k \lambda_{i}(\1_{\{\hat S_n\le h_\alpha(t+u_{i})\}}-\mmp\{\hat S_{n}\le h_{\alpha}(t+u_{i})\}),$$ the left-hand side in \eqref{eq:Cramer-Wold device} is an infinite sum of independent centered random variables with finite second moments. Hence, by the Lindeberg-Feller theorem (see, for instance, Theorem 3.4.5 on p.~129 in \cite{Durrett:2010}), \eqref{eq:Cramer-Wold device} follows if we can show that
\begin{multline}\label{eq:mgale CLT1}
\lim_{t\to\infty}\me\Big[\Big(\sum_{i=1}^{k} \lambda_{i} Z(t,u_{i})\Big)^2\Big]\ =\ \me\Big[\Big(\sum_{i=1}^{k} \lambda_i X_{\alpha}(u_{i})\Big)^2\Big]=\ \sum_{i=1}^k \lambda_{i}^{2}\, {\rm Var}\,X_{\alpha}(u_{i})\,\\+\,2\sum_{1\le i<j\le k}\lambda_{i} \lambda_{j}\,{\rm Cov}\,(X_{\alpha}(u_{i}), X_{\alpha}(u_{j}))
\end{multline}
and
\begin{multline}\label{eq:mgale CLT2}
\lim_{t\to\infty}\eee^{-t}\sum_{n\ge 1}\me \Big[\Big(\sum_{i=1}^k \lambda_i (\1_{\{\hat S_{n}\le h_{\alpha}(t+u_{i})\}}-\mmp\{\hat S_{n}\le h_{\alpha}(t+u_{i})\})\Big)^2\\\times \1_{\{|\sum_{i=1}^{k}\lambda_i(\1_{\{\hat S_{n}\le h_{\alpha}(t+u_{i})\}}-\,\mmp\{\hat S_{n}\le h_{\alpha}(t+u_{i})\})|>\varepsilon \eee^{t/2}\}}\Big]\ =\ 0
\end{multline}
for all $\varepsilon>0$. Relation \eqref{eq:mgale CLT1} follows from Lemma \ref{lem:cov}. In view of $$\sum_{i=1}^k \lambda _i|\1_{\{\hat S_{n}\le h_{\alpha}(t+u)\}}-\mmp\{\hat S_{n}\le h_{\alpha}(t+u)\}|\le k\max_{1\leq i\leq k}\,|\lambda_i|\quad\text{a.s.},$$ the indicator in \eqref{eq:mgale CLT2} is equal to $0$ for sufficiently large $t$. Hence, \eqref{eq:mgale CLT2} also holds.

The proof of \eqref{eq:Cramer-Wold device} is complete.

Our next task is to check that the family of distributions of the processes $((Z(t,u))_{u\in\mr})_{t>0}$ is tight on the Skorokhod space $D[-A,\,A]$ for any fixed $A>0$. Our proof below  is a simplified version of Theorem 2.1 in \cite{Alsmeyer+Iksanov+Kabluchko:2025} dealing with the situation in which $\me[\xi]<\infty$. In particular, we demonstrate that Lipschitz continuity of $V$ required in that theorem is not actually needed.

By Theorem 15.6 in \cite{Billingsley:1968} and the remark following it, it is enough to show that
\begin{equation}\label{eq:tight}
\me\big[(Z(t,v)-Z(t,u))^2 (Z(t,w)-Z(t,v))^2\big]\leq C (\eee^w-\eee^u)^2
\end{equation}
for all $u<v<w$ in the interval $[-A,\,A]$ and large $t>0$.

For $n\in\mn$, introduce the following Bernoulli random variables
\begin{equation*} 
L_n:= \1_{\{h_\alpha(t+u) <\hat S_n \leq h_\alpha(t+v)\}}, \qquad   M_n:= \1_{\{h_\alpha(t+v)<\hat S_n \leq h_\alpha(t+w)\}}
\end{equation*}
and their centered versions
$$
\widetilde L_n:= L_n - \me[L_n], \qquad \widetilde M_n := M_n - \me[M_n].
$$
Put
$$
q_n:= \mmp\{L_n = 1\} = \me[L_n], \qquad  z_n:= \mmp\{M_n = 1\} = \me[M_n].
$$

It is known in the case (D), see, for instance, the equivalence (8.6.1) $\Leftrightarrow$ (8.6.3) on p.~361 in \cite{Bingham+Goldie+Teugels:1989}, that
\begin{equation}\label{eq:first moment}
\me [\hat N(t)]=V(t)~\sim~\frac{1}{\Gamma(1-\alpha)\Gamma(1+\alpha)}(\mmp\{\xi>t\})^{-1},\quad t\to\infty.
\end{equation}
This together with the definition of $h_\alpha$ entails $\lim_{t\to\infty}\eee^{-t}(V(h_\alpha(t+x))-V(h_\alpha(t+y)))=\eee^x-\eee^y$ uniformly in $x,y\in [-A,\,A]$. Hence,
\begin{multline}\label{eq:q_k_est}
\sum_{n\geq 1}q_n =V(h_\alpha(t+v))-V(h_\alpha(t+u))\leq C_1 (\eee^v-\eee^u)\eee^t \quad\text{and}\\ \sum_{n\geq 1}z_n=V(h_\alpha(t+w))-V(h_\alpha(t+v))\leq C_1 (\eee^w-\eee^v)\eee^t 
\end{multline}
for all $u<v<w$ in the interval $[-A,\,A]$, large $t>0$ and a constant $C_1>0$. Observe that
$$\eee^{t/2} (Z(t,v)-Z(t,u))= \sum_{n\geq 1} \widetilde L_n\quad \text{and}\quad \eee^{t/2} (Z(t,w)-Z(t,v))= \sum_{n\geq 1} \widetilde M_n,$$
which implies that \eqref{eq:tight} is equivalent to
$$\me \Big[\Big(\sum_{n_1\geq 1}\widetilde L_{n_1}\Big)^2 \Big(\sum_{n_2\geq 1} \widetilde M_{n_2}\Big)^2\Big]=\sum_{n_1,n_2,n_3,n_4\geq 1} \me \left[\widetilde L_{n_1}\widetilde L_{n_3}\widetilde M_{n_2}\widetilde M_{n_4} \right]  \leq C (\eee^w-\eee^u)^2 \eee^{2t}$$
for all $u<v<w$ in the interval $[-A,\,A]$ and large $t>0$. If $n_1$ is not equal to any of the numbers $n_2$, $n_3$ or $n_4$, then $\widetilde L_{n_1}$ is independent of the vector $(\widetilde L_{n_3}, \widetilde M_{n_2}, \widetilde M_{n_4})$, whence $\me [\widetilde L_{n_1}\widetilde L_{n_3}\widetilde M_{n_2}\widetilde M_{n_4}] = 0$. More generally, if one of the numbers $n_1$, $n_2$, $n_3$ or $n_4$ is not equal to any of the remaining ones, then the expectation vanishes. In the following, we shall consider collections $(n_1,n_2,n_3,n_4)$ in which every number is equal to some other number.

\vspace*{2mm}
\noindent {\sc Case $n_1\neq n_3$.} Then, either $n_2 =n_1$ and $n_4 =n_3$, or $n_2 =n_3$ and $n_4 =n_1$. To explain the idea we only treat the first situation. The corresponding contribution is
$$\sum_{n_1\neq n_3} \me \left[\widetilde L_{n_1}\widetilde L_{n_3}\widetilde M_{n_1}\widetilde M_{n_3}\right]=\sum_{n_1\neq n_3} \me \left[\widetilde L_{n_1}\widetilde M_{n_1}\right] \me \left[ \widetilde L_{n_3}\widetilde M_{n_3}\right].$$ Since $L_{n_1}$ and $M_{n_1}$ cannot be equal to $1$ at the same time, $$\me \left[\widetilde L_{n_1}\widetilde M_{n_1}\right]=-\me[L_{n_1}]\me[M_{n_1}]=-q_{n_1}z_{n_1}.$$
Analogously, $\me \Big[\widetilde L_{n_3}\widetilde M_{n_3}\Big]=-q_{n_3}z_{n_3}$ and thereupon $$
\sum_{n_1\neq n_3} \me \left[\widetilde L_{n_1}\widetilde L_{n_3}\widetilde M_{n_1}\widetilde M_{n_3}\right]=\sum_{n_1\neq n_3} q_{n_1}z_{n_1}q_{n_3}z_{n_3}\leq \sum_{n_1\geq 1}q_{n_1}\sum_{n_2\geq 1}z_{n_2}.
$$
Invoking \eqref{eq:q_k_est} we obtain
\begin{equation}\label{eq:est_sum_q_r}
\sum_{n_1\geq 1} q_{n_1}\sum_{n_2\geq 1}z_{n_2}\leq C_1^2 (\eee^w-\eee^u)^2 \eee^{2t}
\end{equation}
for all $u<v<w$ in the interval $[-A,\,A]$ and large $t>0$.

\noindent {\sc Case $n_1 = n_3$}. We also assume that $n_2= n_4$, for otherwise the expectation $\me[\widetilde L_{n_1}\widetilde L_{n_3}\widetilde M_{n_2}\widetilde M_{n_4}]$ vanishes. The corresponding contribution is given by
\begin{multline*}
\sum_{n_1, n_2\geq 1}\me \left[\widetilde L_{n_1} \widetilde L_{n_1} \widetilde M_{n_2}\widetilde M_{n_2}\right]=\sum_{n_1\neq n_2} \me \left[\widetilde L_{n_1}^2 \right] \me \left[\widetilde M_{n_2}^2\right]+\sum_{n\geq 1}\me \left[\widetilde L_n^2\widetilde M_n^2\right]\\ \leq  \me\Big[\sum_{n_1\neq n_2} q_{n_1}z_{n_2}+ 2\sum_{n\geq 1} q_n z_n\Big] \leq 2 \sum_{n_1\geq 1 }q_{n_1}\sum_{n_2\geq 1}z_{n_2}\leq 2C_1^2 (\eee^w-\eee^u)^2 \eee^{2t}
\end{multline*}
for all $u<v<w$ in the interval $[-A,\,A]$ and large $t>0$. The first equality stems from the fact that $L_n$ and $M_n$ cannot be equal to $1$ simultaneously. The last inequality is justified by~\eqref{eq:est_sum_q_r}. The first inequality follows from $\me[\widetilde L_{n_1}^2]=q_{n_1}(1-q_{n_1})\leq q_{n_1}$, $\me[\widetilde M_{n_2}^2]=z_{n_2}(1-z_{n_2})\leq z_{n_2}$ and
\begin{multline*}
\me\left[\widetilde L_n^2  \widetilde M_n^2\right]=q_n (1-q_n)^2(-z_n)^2 +z_n(1-z_n)^2 (-q_n)^2 + (1-q_n-z_n)(-q_n)^2(-z_n)^2\\=q_nz_n(q_n+z_n-3q_nz_n)\leq 2 q_nz_n.
\end{multline*}
This finishes  the proof of tightness.

The proof of Theorem \ref{main} is complete.
\end{proof}

\section{Proof of Theorem \ref{main1}}

We need an auxiliary result. 

\begin{lemma}\label{var_12}
The variance of $\hat{N}(t)$ exhibits the following asymptotic behavior, as $t \to \infty$,
\begin{equation*} 
{\rm Var}[\hat{N}(t)]~ \sim~ \left(\frac{\sigma^2 t}{\mu^3 \pi} \right)^{1/2} \quad \text{under (A)},
\end{equation*}
\begin{equation*} 
{\rm Var}[\hat{N}(t)]~ \sim~ \left(\frac{1}{\mu^3 \pi} \right)^{1/2} c_2(t) \quad  \text{under (B)},
\end{equation*}
\begin{equation*}
{\rm Var}[\hat{N}(t)]~ \sim~ \frac{\Gamma(1-1/\alpha)}{\mu^{1+1/\alpha}\pi}\Big(\frac{-2\Gamma(2-\alpha)\cos(\pi \alpha/2)}{\alpha-1}\Big)^{1/\alpha}c_\alpha(t) \quad  \text{under (C)},
\end{equation*}
and
\begin{equation*}
{\rm Var}[\hat{N}(t)]~ \sim~ \frac{\int_0^\infty \mmp\{Z_\alpha > y\}\mmp\{Z_\alpha\leq y\}{\rm d}y}{\mmp\{\xi>t\}} \quad  \text{under (D)}.
\end{equation*}
\end{lemma}

The first three asymptotic relations can be found in Corollary 5.3 of \cite{Alsmeyer+Iksanov+Kabluchko:2025}. The last asymptotic relation (in the case (D)) follows from Lemma \ref{lem:cov} upon setting $u=v=0$ and replacing $h_\alpha(t)$ by $t$.

We prove Theorem \ref{main1} by an application of Theorem 1.6 in \cite{Buraczewski+Iksanov+Kotelnikova:2025}. To this end, we need some preparations.

Let $(C_1(t))_{t\geq 0}$, $(C_2(t))_{t\geq 0},\ldots$ be independent families of events defined on a common probability space. Assume that, for each $t\geq 0$, $\sum_{k\geq 1}\mmp (C_k(t))<\infty$ and put $X(t):=\sum_{k\geq 1}\1_{C_k(t)}$. By the Borel-Cantelli lemma, $X(t)<\infty$ a.s. Observe that ${\rm Var}\,[X(t)]=\sum_{k\geq 1}\mmp(C_k(t))(1-\mmp(C_k(t)))\leq \sum_{k\geq 1}\mmp(C_k(t))=\me [X(t)]<\infty$. With a view towards formulating Theorem 1.6 in \cite{Buraczewski+Iksanov+Kotelnikova:2025} we introduce a number of assumptions.

\noindent (A1) $\lim_{t\to\infty}{\rm Var}[X(t)]=\infty$.

\noindent (A2) There exists a nondecreasing function $a_0$ such that ${\rm Var}[X(t)]\sim a_0(t)$ as $t\to\infty$.

\noindent (A3) There exists $\nu^\ast\geq 1$ such that $\me [X(t)]=O(({\rm Var}[X(t)])^{\nu^\ast})$ as $t\to\infty$. Put
\begin{equation}\label{eq:infim}
	\nu:=\inf\{\nu^\ast\ge 1: \me[X(t)]=O(({\rm Var}[X(t)])^{\nu^\ast})\}.
\end{equation}
If $\nu=1$, we assume that either $t\mapsto \me [X(t)]$ is eventually continuous or $${\lim\inf}_{t\to\infty}(\log \me[X(t-1)]/\log \me[X(t)])>0$$ and that
\begin{equation}\label{eq:infim2}
\me [X(t)]/{\rm Var}[X(t)]=O(f_q({\rm Var}[X(t)])),\quad t\to\infty,
\end{equation}
where $f_q(t)=(\log t)^{q}\mathcal{L}(\log t)$ for some $q\geq 0$ with $\mathcal{L}$ slowly varying at $\infty$. If $q>0$, we assume that $\me [X(t)]/{\rm Var}[X(t)]\neq O(f_s(a(t)))$ for $s\in (0,q)$.

\noindent (A4) For each $k\in\mn$ and $t>s>0$, $C_k(s)\subseteq C_k(t)$.

For each $\varrho\in (0,1)$, put
\begin{equation*}
\nu_\varrho:=\nu+\varrho\quad \text{if}~\nu>1\quad\text{and}\quad q_\varrho:=q+\varrho \quad \text{if}~\nu=1.
\end{equation*}
Assuming (A3), fix any $\kappa\in (0,1)$, any $\varrho\in (0,1)$ and put
\begin{equation*}
t_n=t_n(\kappa, \nu):=\inf\{t>0: \me[X(t)]>v_n(\kappa,\nu)\}
\end{equation*}
for $n\in\mn$, where $v_n(\kappa, 1)=v_n(\kappa, 1, q, \varrho)=\exp(n^{(1-\kappa)/(q_\varrho+1)})$ and $v_n(\kappa,\nu)=v_n(\kappa, \nu, \varrho)=n^{\nu_\varrho(1-\kappa)/(\nu_\varrho-1)}$ for $\nu>1$.

\noindent (A5) For each positive $\kappa$ sufficiently close to $0$ and for each $n$ large enough, there exists $A>1$ and a partition $t_n=t_{0,\,n}<t_{1,\,n}<\ldots<t_{j,\,n}=t_{n+1}$ with $j=j_n$ satisfying $$1\leq \me [X(t_{k,\,n})]-\me [X(t_{k-1,\,n})]\leq A,\quad 1\leq k\leq j$$ and, for all $\varepsilon>0$, $\big(j_n\exp(-\varepsilon ({\rm Var}[X(t_n)])^{1/2})\big)$ is a summable sequence.

Assuming (A1) and (A3), fix any $\gamma>0$ and put
\begin{equation*}
\tau_n=\tau_n(\gamma,\nu):=\inf\{t>0: {\rm Var}[X(t)]>w_n(\gamma,\nu)\}
\end{equation*}
for large $n\in\mn$ with $\nu$ as given in \eqref{eq:infim}. Here, with $q$ as given in \eqref{eq:infim2}, $w_n(\gamma, 1)=w_n(\gamma, 1, q)=\exp(n^{(1+\gamma)/(q+1)})$ if $\nu=1$ and $w_n(\gamma, \nu)=n^{(1+\gamma)/(\nu-1)}$ if $\nu>1$.

Here are the remaining assumptions.

\noindent (B1) The function $t\mapsto {\rm Var}[X(t)]$ is eventually continuous or $\lim_{t\to\infty}({\rm Var}[X(t-1)]/{\rm Var}[X(t)])=1$ if $\nu>1$ and $\lim_{t\to\infty}(\log {\rm Var}[X(t-1)]/\log {\rm Var}[X(t)])=1$ if $\nu=1$.

\noindent (B2) There exist $s_0>0$, $\varsigma_0>0$ and a family $(R_{\varsigma}(t))_{0<\varsigma < \varsigma_0, t>s_0}$ of sets of positive integers satisfying the following two conditions: for each $\gamma>0$ close to $0$ and all $0<\varsigma < \varsigma_0$ there exists $n_0=n_0(\varsigma, \gamma)\in\mn$ such that the sets $R_\varsigma(\tau_{n_0}(\gamma,\mu))$, $R_\varsigma(\tau_{n_0+1}(\gamma,\mu)),\ldots$ are disjoint; and
\begin{equation*} 
\lim_{t\to\infty}\frac{{\rm Var}\Big(\sum_{k\in	R_\varsigma(t)}\1_{A_k(t)}\Big)}{{\rm Var}[X(t)]}=1-\varsigma.
\end{equation*}

With these at hand we are ready to formulate Theorem 1.6 in \cite{Buraczewski+Iksanov+Kotelnikova:2025}.
\begin{prop}\label{prop:Buraetal}
Suppose (A1)-(A5), (B1) and (B2). Then, with $\nu\geq 1$ and $q\geq 0$ as defined in \eqref{eq:infim} and \eqref{eq:infim2}, respectively, $$\limsup_{t\to\infty}\frac{X(t)-\me[X(t)]}{(2(q+1){\rm Var}[X(t)]\log\log {\rm Var}[X(t)])^{1/2}}=1\quad\text{{\rm a.s.}}$$ if $\nu=1$ and $$\limsup_{t\to\infty}\frac{X(t)-\me[X(t)]}{(2(\nu-1){\rm Var}[X(t)]\log {\rm Var}[X(t)])^{1/2}}=1\quad\text{{\rm a.s.}}$$ if $\nu>1$.

The corresponding lower limits are equal to $-1$ a.s.
\end{prop}

\begin{proof}[Proof of Theorem \ref{main1}]
We intend to check that the conditions (A1)-(A5), (B1) and (B2) hold true.

\noindent {\sc Condition (A1)} holds\footnote{More generally, it was recently proved in Proposition 6 of \cite{Buraczewski+Iksanov+Marynych:2025} that $\lim_{t\to\infty}{\rm Var}[\hat N(t)]=\infty$ under the sole assumption that the distribution of $\xi$ is nondegenerate.} by Lemma \ref{var_12}.

\noindent {\sc Condition (A2)} is secured by Lemma \ref{var_12} in the cases (A) and (D). In the cases (B) and (C), $c_\alpha(t)\sim a_0(t)$, where $a_0$ is a nondecreasing function defined by $$a_0(t):=\inf\{s\geq 0: 1/\mmp\{\xi>s\}>t\},\quad t>1.$$ Another appeal to Lemma \ref{var_12} justifies condition (A2) in the cases (B) and (C).

\noindent {\sc Condition (A3)}. In the cases (A), (B) and (C), $\me [\hat N(t)]=V(t)\sim \mu^{-1}t$ as $t\to\infty$ by the elementary renewal theorem. This together with Lemma \ref{var_12} enables us  to conclude that $\nu=2$ in the cases (A) and (B) and that $\nu=\alpha$ in the case (C). Formula \eqref{eq:first moment} in combination with Lemma \ref{var_12} entails $\nu=1$,  $q=0$ and $\lim_{t\to\infty}(\log \me [\hat N(t-1)]/\log \me [\hat N(t)])=1$.

\noindent {\sc Condition (A4)} holds trivially.

\noindent {\sc Condition (A5)}. Pick any $m \in \mn$ satisfying $m > 2 + \frac{\me[\eee^{-\xi}]}{1-\me[\eee^{-\xi}]}$. We shall work with the function $W$ defined in \eqref{def_v}. For large enough $n$, put $j=j_n=\lfloor m^{-1}(V(t_{n+1})-V(t_n))\rfloor$ and construct the partition $t_n=t_{0,n}<t_{1,n}<\ldots<t_{j,n}=t_{n+1}$ via $W(t_{k,n})-W(t_{k-1,n})= m$ for $k\in\mn$, $k\leq j-1$. Then, necessarily $W(t_{j,n})-W(t_{j-1,n}) \in [m, 2m)$. The construction is possible, for, by Lemma \ref{uv_mon}, $W$ is continuous and strictly increasing for all arguments large enough.

Pick $\varepsilon>0$ satisfying $m \geq 2 + \frac{\me[\eee^{-\xi}]}{1-\me[\eee^{-\xi}]} + \varepsilon$. Write
\begin{equation*}
0\leq V(t)-W(t) = V(t)- \me  [V(t-\eta)\1_{\{\eta \leq t\}}] = \me[(V(t)-V(t-\eta))\1_{\{\eta \leq t\}}]+V(t)\eee^{-t}.
\end{equation*}
The function $V+1$ is subadditive, whence
\begin{equation*}
\me [(V(t)-V(t - \eta))\1_{\{\eta \leq t\}}] \leq \me [(1 + V(\eta)) \1_{\{\eta \leq t\}}] \leq 1 + \me[V(\eta)] = \frac{1}{1-\me[\eee^{-\xi}]}.
\end{equation*}
By the elementary renewal theorem, $V(t)\eee^{-t}\leq \varepsilon$ for large $t$. Thus, for large $t$,
\begin{equation*}
0\leq V(t)-W(t) \leq \frac{1}{1-\me[\eee^{-\xi}]} + \varepsilon.
\end{equation*}
This in combination with
\begin{equation*}
2m \geq W(t_{k,n})-W(t_{k-1,n})\geq W(t_{k,n})-V(t_{k-1,n})= V(t_{k,n}) - V(t_{k-1,n}) - (V(t_{k,n})-W(t_{k,n})),
\end{equation*}
which holds for large $n$ and $1\leq k\leq j_n$, proves $$V(t_{k,n}) - V(t_{k-1,n}) \leq 2m +\frac{1}{1-\me[\eee^{-\xi}]} + \varepsilon.$$ As for the lower bound, using
\begin{equation*}
m \leq W(t_{k,n})-W(t_{k-1,n}) \leq V(t_{k,n})-W(t_{k-1,n})= V(t_{k,n}) - V(t_{k-1,n}) +  (V(t_{k-1,n})-W(t_{k-1,n}))
\end{equation*}
we obtain $$V(t_{k,n}) - V(t_{k-1,n})\geq m -\Big(\frac{1}{1-\me[\eee^{-\xi}]} + \varepsilon\Big)\geq 1.$$

Now we have to check that, for all $\varepsilon>0$, the sequence $(j_n \exp(-\varepsilon({\rm Var}[\hat{N}(t_n)])^{1/2}))$ is summable. We start by discussing the cases (A), (B) and (C). The sequence $(j_n)$ then exhibits at most polynomial growth: $$j_n\leq m^{-1}(W(t_{n+1})-W(t_n))\leq m^{-1}V(t_{n+1})~\sim~ m^{-1}n^{\nu_\varrho(1-\kappa)/(\nu_\varrho-1)},\quad n\to\infty,$$ where $\nu=2$ in the cases (A) and (B) and $\nu=\alpha$ in the case (C). Since $t_n\sim \mu n^{\nu_\varrho(1-\kappa)/(\nu_\varrho-1)}$ and, by Lemma \ref{var_12}, $t\mapsto {\rm Var}[\hat N(t)]$ exhibits the polynomial growth, we conclude that $n\mapsto {\rm Var}[\hat N(t_n)]$ grows polynomially. Thus, the sequence in focus is indeed summable in the cases (A), (B) and (C). In the case (D), relation \eqref{eq:first moment} ensures that $j_n$ is of at most polynomial growth in the argument $\exp((n+1)^{(1-\kappa)/(\varrho+1)})\sim \exp(n^{(1-\kappa)/(\varrho+1)})=t_n$. Also, by Lemma \ref{var_12}, $n\mapsto {\rm Var}[\hat N(t_n)]$ is of polynomial growth in $t_n$. Thus, the sequence $(j_n \exp(-\varepsilon({\rm Var}[\hat{N}(t_n)])^{1/2}))$ is summable in the case (D), too.

\noindent {\sc Condition (B1)}. By Lemma \ref{var_12}, the relation $\lim_{t\to\infty}({\rm Var}[\hat N(t-1)]/{\rm Var}[\hat N(t)])=1$ holds in the cases (A), (B) and (C) (in which $\nu>1$) and the relation $\lim_{t\to\infty}(\log {\rm Var}[\hat N(t-1)]/\log {\rm Var}[\hat N(t)])=1$ holds in the case (D) (in which $\nu=1$).

\noindent {\sc Condition (B2)}. We first treat the cases (A), (B) and (C). Recall that the functions $c_\alpha$ were defined in Section \ref{sect:prelim}, for instance, $c_2(t)=\sigma t^{1/2}$ in the case (A). 

Fix any $x>0$ and put $$c(t)=\lfloor \mu^{-1}t-\mu^{-1-1/\alpha}c_\alpha(t)x \rfloor\quad\text{and}\quad d(t)=\lfloor \mu^{-1}t+\mu^{-1-1/\alpha}c_\alpha(t)x \rfloor$$ for large $t$ satisfying  $\mu^{-1}t-\mu^{-1-1/\alpha}c_\alpha(t)x\geq 1$. Write
\begin{multline*}
{\rm Var}\Big[\sum_{k=c(t)}^{d(t)-1}\1_{\{\hat S_k \leq t\}}\Big]=\int_{c(t)}^{d(t)}\mmp\{S_{\lfloor x\rfloor}\leq t\}\mmp\{ S_{\lfloor x\rfloor} > t \}{\rm d}x\\=\mu^{-1-1/\alpha} c_\alpha(t) \int_{\frac{c(t)-\mu^{-1}t}{ \mu^{-1-1/\alpha}c_\alpha(t)}}^{\frac{d(t)-\mu^{-1}t}{\mu^{-1-1/\alpha}c_\alpha(t)}} \mmp\{S_{\lfloor \mu^{-1}t + \mu^{-1-1/\alpha}c_\alpha(t) y\rfloor} \leq t\}\mmp\{S_{\lfloor \mu^{-1}t + \mu^{-1-1/\alpha} c_\alpha(t)y\rfloor} > t\}{\rm d}y.
\end{multline*}
According to \eqref{eq:basic conv}, for each fixed $y\in\mr$,
		\begin{equation*}
			\lim_{t\to\infty}\mmp\{S_{\lfloor \mu^{-1}t +\mu^{-1-1/\alpha}c_\alpha(t)y\rfloor} \leq t\}= \lim_{t\to\infty}\mmp\{\tau(t) > \lfloor \mu^{-1}t +\mu^{-1-1/\alpha}c_\alpha(t)y\rfloor \}= \mmp\{Z_\alpha >y\}
		\end{equation*}
and
\begin{multline*}
\lim_{t\to\infty}\mmp\{S_{\lfloor \mu^{-1}t +\mu^{-1-1/\alpha}c_\alpha(t)y\rfloor} > t\}= \lim_{t\to\infty}\mmp\{\tau(t) \leq \lfloor \mu^{-1}t +\mu^{-1-1/\alpha}c_\alpha(t)y\rfloor \}= \mmp\{Z_\alpha \leq y\}.
\end{multline*}
Invoking the Lebesgue dominated convergence theorem we infer
		\begin{multline*}
			\lim_{t\to\infty}\int_{\frac{c(t)-\mu^{-1}t}{\mu^{-1-1/\alpha}c_\alpha(t)}}^{\frac{d(t)-\mu^{-1}t}{\mu^{-1-1/\alpha}c_\alpha(t)}} \mmp\{S_{\lfloor \mu^{-1}t +\mu^{-1-1/\alpha}c_\alpha(t)y\rfloor} \leq t\}\mmp\{S_{\lfloor \mu^{-1}t +\mu^{-1-1/\alpha}c_\alpha(t)y\rfloor} > t\}{\rm d}y \\= \int_{-x}^{x} \mmp\{Z_\alpha>y\}\mmp\{Z_\alpha \leq y\}{\rm d}y.
		\end{multline*}
Thus, $${\rm Var}\Big[\sum_{k=c(t)}^{d(t)} \1_{\{\hat S_k \leq t\}}\Big]~\sim~ \mu^{-1-1/\alpha}c_\alpha(t) \int_{-x}^{x} \mmp\{Z_\alpha>y\}\mmp\{Z_\alpha \leq y\}{\rm d}y,\quad t\to\infty.$$ This ensures that given $\varsigma\in (0,1)$ there exists the unique value $x > 0$ such that, as $t\to\infty$, $${\rm Var}\Big[\sum_{k=c(t)}^{d(t)}\1_{\{\hat S_k \leq t\}}\Big]~\sim~ (1 - \varsigma)\mu^{-1-1/\alpha}c_\alpha(t) \int_\mr \mmp\{Z_\alpha>y\}\mmp\{Z_\alpha \leq y\}{\rm d}y~\sim~ (1-\varsigma){\rm Var}[\hat{N}(t)].$$ We have shown that in the cases (A), (B) and (C) the second part of condition (B2) holds.

Now we check the first part of condition (B2) in the cases (A), (B) and (C). To this end, it suffices to show that $\lim_{t\to\infty}(c(\tau_{n+1}) - d(\tau_n)) = +\infty$ or, equivalently, that $$\lim_{n\to\infty}\left( \mu^{-1}(\tau_{n+1}-\tau_n)-\mu^{-1-1/\alpha}x\left(c_\alpha(\tau_{n+1}\right)+c_\alpha(\tau_n))\right) = +\infty.$$ Recall that, for $\alpha\in (1,2]$, the function $c_\alpha$ varies regularly at $\infty$ of index $1/\alpha$. Put $K(s) := \inf\{t\geq 0: {\rm Var}[\hat{N}(t)]> s^{(1+\gamma)/(\nu-1)}\}$ for $s>0$, where $\nu=2$ in the cases (A) and (B) and $\nu=\alpha$ in the case (C). By Lemma \ref{var_12}, the function $t\mapsto {\rm Var} [\hat{N}(t)]$ is regularly varying at $\infty$ of index $1/\alpha$. Hence, by Theorem 1.5.12 in \cite{Bingham+Goldie+Teugels:1989}, $K(s)=s^\rho L(s)$ for $s>0$, where $\rho:=(1+\gamma)\alpha/(\nu-1)$ and $L$ varies slowly at $\infty$. Since $\tau_n = K(n)$ for $n \in \mn$, we conclude that $\lim_{n\to\infty}(c_\alpha(\tau_{n+1})/c_\alpha(\tau_n))=1$ and $\lim_{n\to\infty}c_\alpha(\tau_n)=+\infty$. Thus, it is enough to prove that $\lim_{n\to\infty}(\tau_{n+1}-\tau_n)/c_\alpha(\tau_n)=+\infty$.

To check this, we use a representation $$\tau_{n+1}-\tau_n = ((n+1)^\rho - n^\rho)L(n+1) + n^\rho(L(n+1)-L(n)).$$ Observe that $$\frac{((n+1)^\rho - n^\rho)L(n+1)}{c_\alpha(\tau_n)}~\sim~ \frac{\rho n^{\rho-1} L(n)}{c_\alpha(\tau_n)}~\sim~ n^{\frac{(1+\gamma)\alpha}{\nu-1} - \frac{(1+\gamma)}{\nu-1}-1}L_1(n) = n^\gamma L_1(n)~\to~\infty,\quad n\to\infty.$$ Here, $L_1$ is a slowly varying function defined by the ratio of $\rho L$ and the slowly varying factor of $c_\alpha(\tau_n)$. When working with the second summand in the representation of $\tau_{n+1}-\tau(n)$ we can and do assume that the function $L$ is differentiable and $\lim_{t\to\infty}(tL^\prime(t)/L(t))=0$, see Theorem 1.3.3 in  \cite{Bingham+Goldie+Teugels:1989}. By the mean value theorem for differentiable functions, there exists some $\theta_n \in [n, n+1]$ satisfying $$\frac{n(L(n+1)-L(n))}{L(n)}=\frac{nL^\prime(\theta_n)}{L(n)} =  \frac{\theta_nL^\prime(\theta_n)}{L(\theta_n)}\frac{n}{\theta_n}\frac{L(\theta_n)}{L(n)}~\to~0,\quad n\to\infty.$$ Here, the last factor converges to $1$ as $n\to\infty$ according to the uniform convergence theorem for slowly varying functions (Theorem 1.2.1 in \cite{Bingham+Goldie+Teugels:1989}). Thus, $$\frac{n^\rho(L(n+1)-L(n))}{c_\alpha(\tau_n)}=o(n^\gamma L_1(n)),\quad n\to\infty,$$ and the first part of condition (B2) holds in the cases (A), (B) and (C).

In the case (D), we fix $x \in (0,1)$ and put $c(t)=\lfloor x/\mmp\{\xi>t\}\rfloor$ and $d(t)=\lfloor 1/(x\mmp\{\xi>t\}) \rfloor$ for large $t$ satisfying $x/\mmp\{\xi>t\}\geq 1$. By the same reasoning as above we infer with the help of \eqref{eq:basic conv2},
$${\rm Var}\Big[\sum_{k=c(t)}^{d(t)-1}\1_{\{\hat S_k \leq t\}}\Big]~\sim~ \frac{\int_x^{1/x}\mmp\{S_\alpha > y\} \mmp\{S_\alpha\leq y\}{\rm d}y}{\mmp\{\xi>t\}},\quad t\to\infty.$$ As a consequence, given $\varsigma\in (0,1)$ there exists the unique value $x \in (0,1)$ such that $${\rm Var}\Big[\sum_{k=c(t)}^{d(t)}\1_{\{\hat S_k \leq t\}}\Big]\sim (1-\varsigma)\frac{\int_0^\infty \mmp\{S_\alpha > y\} \mmp\{S_\alpha\leq y\}{\rm d}y}{\mmp\{\xi>t\}}~\sim~(1 - \varsigma){\rm Var} [\hat{N}(t)],\quad t \to \infty.$$ Thus, the second part of condition (B2) holds in the case (D).

To check the first part of condition (B2) in the case (D) (that is, $\nu=1$ and $q=0$) it is enough to show that $\lim_{t\to\infty}(c(\tau_{n+1}) - d(\tau_n)) = +\infty$ or, equivalently,
\begin{equation*}
\lim_{n\to\infty}\left(\frac{x}{\mmp\{\xi>\tau_{n+1}\}}-\frac{1}{x\mmp\{\xi>\tau_n\}}\right) = \lim_{n\to\infty}\frac{1}{\mmp\{\xi>\tau_n\}} \left(\frac{x\mmp\{\xi>\tau_n\}}{\mmp\{\xi>\tau_{n+1}\}}-\frac{1}{x}\right) =  +\infty.
\end{equation*}
In the present case, $\tau_n= \inf \{t>0: {\rm Var}[\hat{N}(t)] > \exp(n^{1+\gamma})\}$ for some $\gamma>0$, so that $\lim_{n\to\infty}(\mmp\{\xi>\tau_n\}/\mmp\{\xi>\tau_{n+1}\})=+\infty$. Hence, the first part of condition (B2) holds in the case (D), too.

We have shown that conditions (A1)-(A5), (B1) and (B2) of Proposition \ref{prop:Buraetal} hold. Now Theorem \ref{main1} follows by an application of Proposition \ref{prop:Buraetal}. We stress that $\hat N(t)$ only satisfies a law of the iterated logarithm in the case (D) (in which $\nu=1$). In the cases (A), (B) and (C) (in which $\nu>1$) the a.s. asymptotic behavior of $\hat N(t)$ is driven by laws of the single logarithm.

\end{proof}

\noindent \textbf{Acknowledgment.} I. Feshchenko gratefully acknowledges the support by the National Research Foundation of Ukraine (project 2023.03/0059 ‘Contribution to modern theory of random series’). A. Iksanov gratefully acknowledges the support by the MOHRSS program (H 20240850).

\end{document}